\documentclass[12pt]{amsart}
\usepackage{geometry}                
\usepackage{graphicx}

\usepackage{amsthm}
\usepackage{amsmath}
\usepackage{amsfonts}
\usepackage{enumerate}
\usepackage{amssymb}
\usepackage{epstopdf}
\usepackage{hyperref}

\newtheorem{theorem}{Theorem}[section]
\newtheorem{proposition}[theorem]{Proposition}
\newtheorem{lemma}[theorem]{Lemma}
\newtheorem{corollary}[theorem]{Corollary}

\theoremstyle{definition}
\newtheorem{definition}[theorem]{Definition}
\newtheorem{remark}[theorem]{Remark}

\author{Gregory R. Chambers}
\address{Department of Mathematics, University of Chicago, Chicago, IL, USA}
\email{chambers@math.uchicago.edu}
\author{Yevgeny Liokumovich}
\address{Department of Mathematics, Imperial College London, UK}
\email{y.liokumovich@imperial.ac.uk}

\begin{document}

\title{Optimal sweepouts of a Riemannian 2-sphere}

\newcommand{\R}{\mathbb R}

\maketitle

\begin{abstract}

Given a sweepout of a Riemannian $2$-sphere which is composed of curves of length less than
$L$, we construct a second sweepout composed of curves of length less than
$L$ which are either constant curves or simple curves.

This result, and the methods used to prove it, have several consequences; we answer a
question of M. Freedman
concerning the existence of min-max embedded geodesics, we partially answer a question due to N. Hingston
and H.-B. Rademacher, and we also extend the results of \cite{CL} concerning converting homotopies to isotopies in an effective way.
\end{abstract}

\section{Introduction}

Let $(S^2, g)$ be a Riemannian 2-sphere and let $\Lambda$
denote the space of all smooth closed curves on $(S^2, g)$.
Let $E(\gamma)= \int_0 ^1 |\gamma'(s)|^2 ds $ denote the energy of a closed curve
$\gamma \in \Lambda$.
A \textit{sweepout} $\{\gamma_t(s) \}_{t \in S^1}$ is a family of closed
curves corresponding to the generator
of $H_1(\Lambda, \mathbb{Z})$.

Consider the following min-max quantity
$$W(S^2,g) = \inf_{\{\gamma_t \}} \sup_t \sqrt{E (\gamma_t) }$$
where the infimum is taken over all sweepouts $\{ \gamma_t \}$ of 
$(S^2,g)$.

By a classical argument going back to Birkhoff \cite{Bir},
$(S^2, g)$ contains a closed
geodesic of length $W(S^2,g)$.



In the 1980s, M. Freedman considered the question of
whether one can construct an embedded geodesic
via min-max methods on the space $\Lambda$.
Apart from being a fundamental question about closed geodesics,
this question also has the following motivation. 
Consider a similar problem for families of
2-dimensional spheres in a homotopy sphere $M$.
If we could replace a sweepout of $M$ by immersed 2-spheres
with a sweepout by embedded 2-spheres, then
by the ambient isotopy theorem it would follow 
that $M$ is diffeomorphic to $S^3$, implying
the Poincar\'{e} conjecture (see \cite{Freed}).

In this paper we give an affirmative answer to
Freedman's question.

\begin{theorem} \label{Freedman}
Every Riemannian 2-sphere $(S^2,g)$ contains an embedded closed geodesic of length
$W(S^2,g)$.
\end{theorem}

A direct way of proving this result would be to start with an
arbitrary sweepout, cut each curve in the family at the points
of self-intersection, and then reglue them so that after a small 
perturbation we obtain a collection of embedded curves.
Cutting and regluing at finitely many points increases the energy
by an arbitrarily small amount. If we could always obtain a sweepout
by embedded curves this way then the min-max argument yields
a sequence of embedded closed curves converging to a (necessarily) embedded
closed geodesic.

This cutting and regluing procedure can be done in many different ways.
One may try to find a way so that the resulting collection of
curves forms a continuous isotopy.
In \cite{Freed}, Freedman showed that such a direct approach fails.
He constructs a family of curve with the property that, no matter how we choose to reglue the
curves, the new family will contain a discontinuity.
In this article, we circumvent this problem by
first performing certain surgeries and, more importantly,
by making use of the idea from \cite{CL} of moving
`back and forth' in time when constructing our homotopy out of 
regluings of curves.  In general, the regluing procedure described above
can produce many different simple curves from a given curve; this latter idea
avoids discontinuities in the resulting isotopy by allowing
it to pass through different regluings of each curve.

Theorem \ref{Freedman} follows from a stronger result
about simple sweepouts of $(S^2,g)$.
We say that a family of curves $\gamma_t$ for $t \in [-1,1]$ is a \textit{simple sweepout}
of $(S^2,g)$ if every curve is either a constant curve, or is a simple curve.
We prove the following theorem:

\begin{theorem} \label{thm:simple_sweepout}
If there exists a sweepout $\{\gamma_t\}$ of $(S^2,g)$
consisting of curves of length less than $L$ 
and energy less than $E$, then there exists a simple sweepout of
$(S^2,g)$ consisting of curves of length less than $L$
and energy less than $E$.
\end{theorem}

Our construction of a simple sweepout consists of two steps. 
The first step is to modify our original sweepout
$\gamma$ so that it begins and ends at a constant curve.
This is accomplished in Section 3 by means of a certain surgery along 
a self-intersection point. In particular, we define a procedure
that involves cutting some of the curves in $\gamma$
at their self-intersection points and assembling the resulting
subcurves into a new noncontractible family of curves that starts and
ends at a constant curve. This is illustrated in 
Figure \ref*{fig:surgery}.

\begin{figure}
   \centering   
    \includegraphics[scale=0.3]{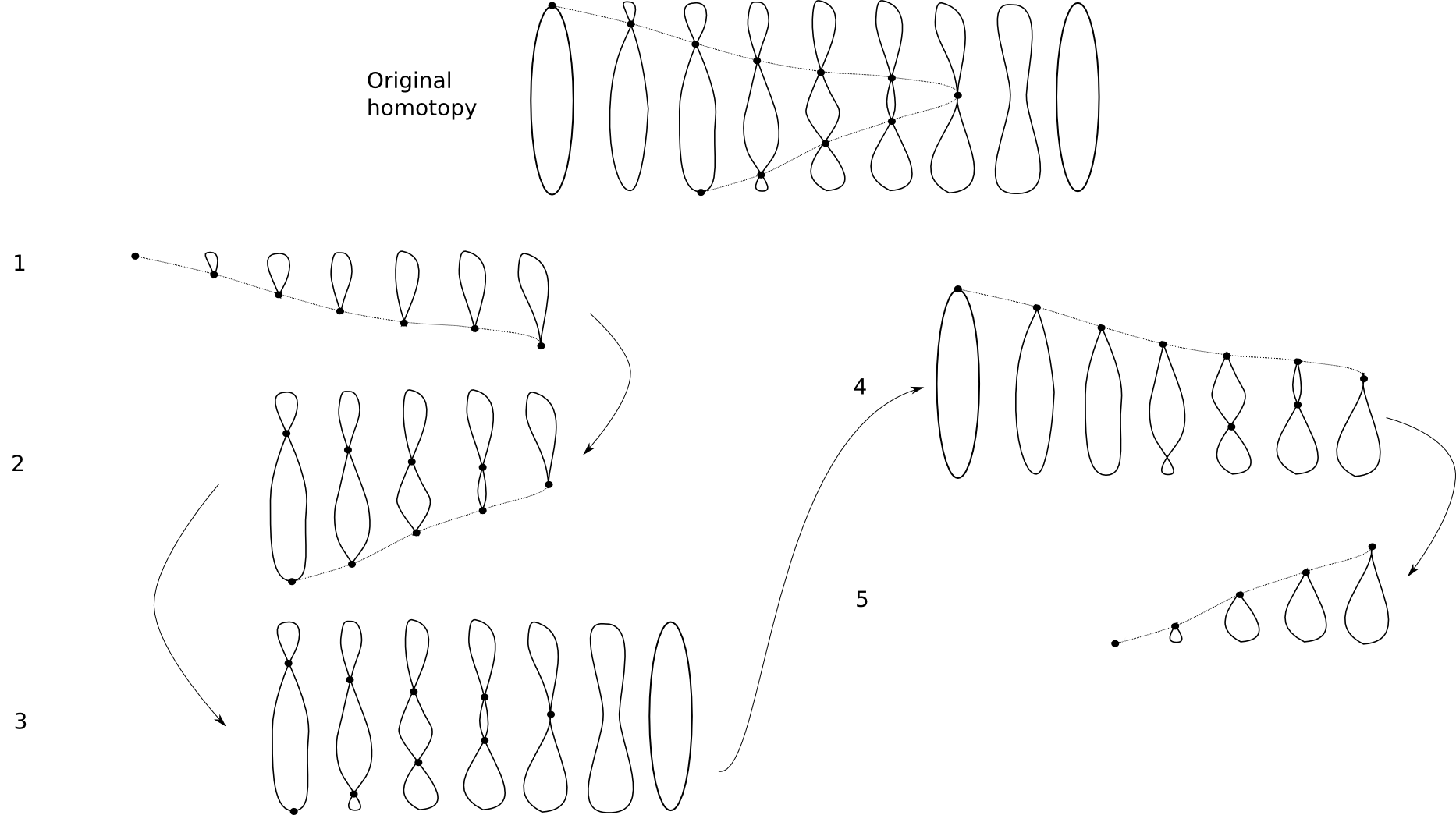}
    \caption{Producing a sweepout with constant curves as endpoints} \label{fig:surgery}

\end{figure}

For a general family of curves such a surgery may not be
possible (see Figure \ref*{fig:replicating} and the discussion at the
end of 3.1). We show, however, that it is possible whenever
the family $\gamma$ is noncontractible. To prove this result,
we derive a certain formula (Proposition \ref*{degree formula}) that relates the degree of
the map $\gamma$ to 
the number of times a small loop (oriented appropriately)
is created or destroyed in this family. 
This formula is a consequence of Whitney's theorem that
the turning number of a curve on the plane does not change under
regular homotopies.

The second step is to remove the self-intersections of each curve in the family.
To accomplish this, we apply Theorem $1.1\prime$ from  
\cite{CL}.  This theorem takes a homotopy between two simple closed curves $\gamma_0$ and $\gamma_1$ through curves
of length (energy) less than $L$ (resp. $E$), and either constructs an isotopy between $\gamma_0$ and $\gamma_1$ or an isotopy between
$\gamma_0$ and $-\gamma_1$ through curves of length (energy) less than $L$ (resp. $E$).  Here, $-\gamma_1$ is $\gamma_1$ with the
opposite orientation.
We give a brief description of this construction as we will need to
use some of its properties in this article.

Consider a closed curve $C$ as a graph with vertices corresponding to self-intersection points
and edges corresponding to arcs of the curve. A redrawing of $C$ is a connected
closed curve obtained by traversing the edges of this graph exactly once
in such a way that the result is simple after a small perturbation.
In \cite{CL} the authors showed that, given a homotopy $\gamma$,
the space of all redrawings of all curves in $\gamma$ is homeomorphic to 
a graph and that this graph
contains a path that connects a redrawing of the initial curve
to a redrawing of the final curve.
We apply the same argument to our family of curves 
that starts and ends at a point. As a result, we obtain a
family that starts and ends at a point and consists of simple
closed curves.
A priori this new family may be contractible in $\Lambda$;
this indeed happens if the degree of the map $\gamma$ is even. In Section 4 we prove 
that, if we start with a family of curves that corresponds to 
a map from torus to $S^2$ of odd degree,
then this construction will produce a
sweepout in which each curve in our modified family consists of arcs
of some curve in the original family.

Our argument works in the same manner if we start from any family of 
closed curves which corresponds to a map from the torus to $S^2$ of odd
degree. In particular, the methods above prove the following result, which partially answers a question of
N. Hingston and H.-B. Rademacher which appeared in \cite{HR} and \cite{BM}.

Let $M$ be a Riemannian manifold.
Given a homology class $X \in H_1 (\Lambda M, \mathbb{Z})$ define a critical level of
$X$ to be the following min-max quantity

$$cr(X) = \inf_{\gamma \in X} \sup_{t \in S^1} \sqrt{E(\gamma_t)} $$
where the infimum runs over all families of curves in the homology class $X$.
It is a standard result in Morse theory that every critical level
corresponds to a closed geodesic on $M$
of length equal to $cr(X)$.
Given a homology class $X$ of the free loop space of $M = (S^n, g)$ of infinite order, N. Hingston and H.-B. Rademacher asked
how $cr(X)$ and $cr(kX)$ are related for each integer $k$. We answer this question for specific values of $X$, $k$, and $n$:

\begin{corollary} \label{cor:hingston}
If $X$ is the generator of $H_1 (\Lambda M, \mathbb{Z})$  and $k$ is odd, then
$$cr(X)=cr(k X).$$
\end{corollary}

In the last section, we use the methods developed in this article to prove a conjecture 
from \cite{CL} about isotopies of curves on a Riemannian 2-surface:

\begin{theorem} \label{thm:isotopy}

Let $M$ be a $2$-dimensional Riemannian manifold (with or without boundary) and
let $\gamma_0$ and $\gamma_1$ be two simple closed curves which are homotopic
through curves of length less than $L$. We have that at least one of the following statements holds:

\begin{enumerate}

\item $\gamma_0$ and $\gamma_1$ are homotopic through simple closed curves of length less than $L$.

\item $\gamma_0$ and $\gamma_1$ are each contractible through simple closed curves
	of length less than $L$.  Here, we mean that all curves except for the final curve are simple.
\end{enumerate}
\end{theorem}

To illustrate this theorem, let $\gamma_0$ be a small contractible loop on a surface $M$ and let
$\gamma_1$ be the same loop with the opposite orientation and suppose they are homotopic through short curves. 
If $M$ is a torus then $\gamma_0$ and $\gamma_1$ are homotopic, but not isotopic. If $M$ is a sphere then $\gamma_0$ and $\gamma_1$ are isotopic, but possibly only through curves of much larger length. 

\vspace{0.1in}

\textbf{Acknowledgments.} 
The authors would like to thank Alexander Nabutovsky and Regina Rotman
for introducing them to the questions studied in this paper, and for 
many useful discussions. They would also like to thank Michael Freedman for
directing them to \cite{Freed}, and for pointing out that Theorem \ref*{Freedman} followed from their results.
Both authors were supported in part by the Government of Ontario
through Ontario Graduate Scholarships.

\section{Reidemeister moves and degrees of maps}

\subsection{Generic sweepouts} \label{sec:generic}

We begin by defining three types of \emph{Reidemeister moves}, shown in
Figure \ref*{fig:reidemeister_moves}.  These are ways in which a 
generic one parameter family of smooth curves may locally self-interact.  As in this figure, we categorize them
as Type 1, Type 2, and Type 3 moves.

\begin{figure}
   \centering   
    \includegraphics[scale=0.8]{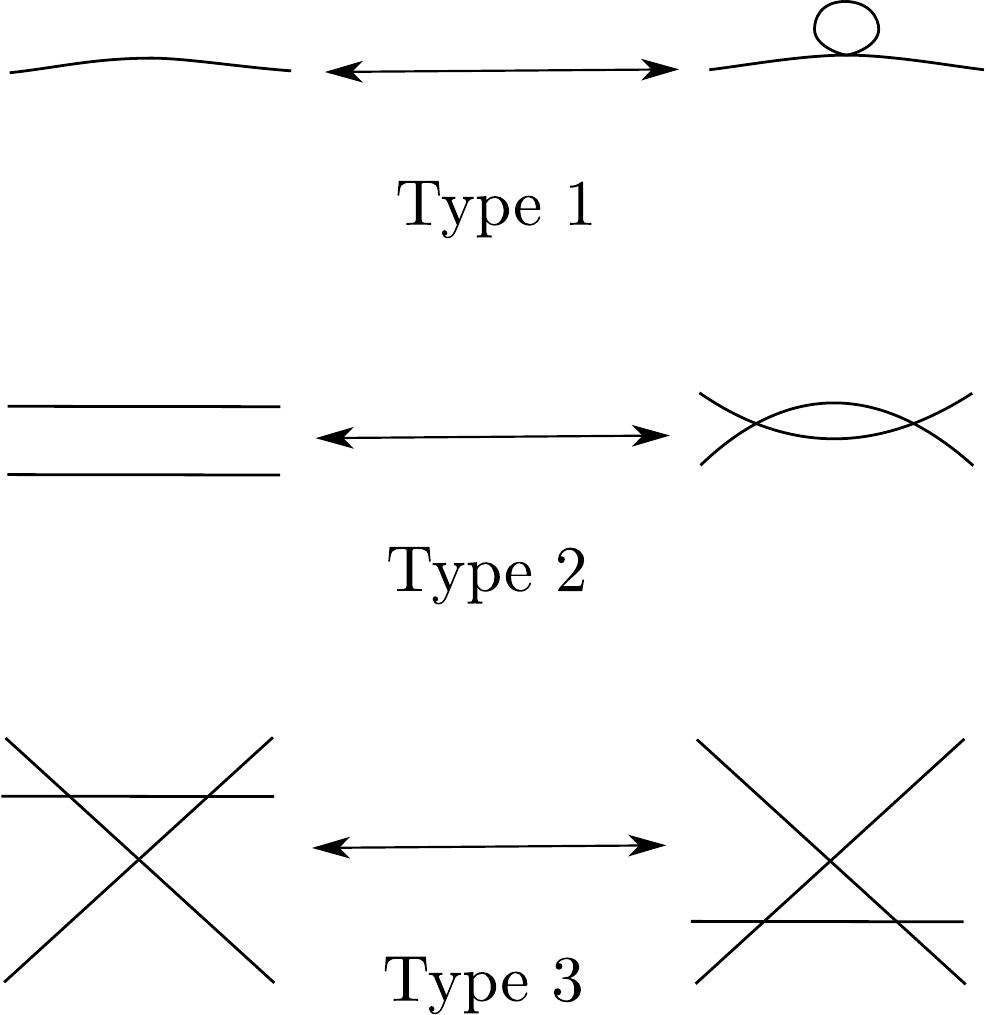}
    \caption{Reidemeister moves} \label{fig:reidemeister_moves}

\end{figure}

Now, fix a map $\gamma: S^1 \rightarrow \Lambda$.  Thom's Multijet Transversality Theorem (see \cite{GG}, \cite[Proposition 2.1]{CL}) implies that we can perturb $\gamma$ to $\widetilde{\gamma}$
so that $\widetilde{\gamma}$ has the following properties:
\begin{enumerate}
  \item  For every $t \in S^1$, $length(\widetilde{\gamma}_t) < L $,
  $E(\widetilde{\gamma}_t)<E$.
  \item  $\widetilde{\gamma}$ is a sweepout.
  \item  For each $t \in S^1$, $\widetilde{\gamma}_t$ contains only isolated self-intersections.  That is, for each self-intersection of each $\widetilde{\gamma}_t$, there is an open
	 ball centered at that self-intersection that contains no other self-intersection.
  \item  $\widetilde{\gamma}$ is composed of a finite set of Reidemeister moves.  That is, there is a finite sequence of points $t_1, \dots, t_n \in S^1$ such that
    exactly one Reidemeister move occurs between $t_i$ and $t_{i+1}$ for each $i \in \{ 1, \dots, n - 1 \}$, and exactly one Reidemeister move
    occurs between $t_n$ and $t_1$.  At each time $t_i$, the curve has transverse intersections only.
\end{enumerate}

Since any smooth map can be perturbed to have this property, we will assume that $\gamma$ has already been put into this form.  Additionally,
for any sweepout $\gamma$, we say that it is \emph{generic} if it has the above properties.  We define a generic homotopy and a generic family $S^1 \rightarrow \Lambda$ of closed curves in analogous ways.


\subsection{Degree and Type 1 Reidemeister moves}

We will say that a continuous family of closed curves
$\{\gamma_t \}_{t \in S^1}$ has degree $d$
if the corresponding map from the torus to $S^2$
has degree $d$.
In this section, we will describe a procedure that assigns either a $+1$ or a $-1$ to
each Type 1 move in a sweepout.  This will be called the sign of the corresponding
Type 1 move.  We will then 
prove a formula relating the sum of these signs to the degree of the sweepout.
This formula is a consequence of Whitney's theorem that 
connected components of the
 space of immersed curves
in $\R^2$ are classified by their turning numbers.

The turning number $T(\gamma)$ of an immersed curve $\gamma \in \R^2$
is defined as the degree of the Gauss map 
sending a point $t \in S^1$ to $\frac{d \gamma}{dt}(t)/ |\frac{d \gamma}{dt}(t)|$.

Fix an orientation on $S^2$.
Let $\gamma$ be a generic homotopy of closed curves on $S^2$.
Suppose a Type 1 move happens at time $t'$.
Note that for a small $\delta>0$,
the curves $\gamma_{t}$ with $t'-\delta \leq t \leq t'+\delta$ are immersed 
everywhere except for a small open disc
where a small loop is created or destroyed. 
Fix a point $p \in S^2$ such that $p$ does not intersect $\gamma_{t}$
for all $t \in [t'-\delta,t'+\delta]$.
Let $St_{p}: S^2 \setminus \{ p \} \rightarrow \R^2$ denote
the stereographic projection with respect to $p$.
We say that the Type 1 move is positive if the turning number
of the curve $St_{p}(\gamma_{t})$ in $S^2 \setminus \{ p \}$
increases by one and we say that it is negative if it decreases by $1$.
Observe that this definition is independent of the choice
of $p$ as long as $\gamma_{t}$ does not intersect 
$p$ for $t'-\delta \leq t \leq t'+\delta$.

\begin{proposition} \label{degree formula}
Let $\gamma: S^1 \rightarrow \Lambda$ be a generic family of closed curves
of degree $d$.
The sum of the signs of all Type 1 moves is equal to $2d$.
\end{proposition}

\begin{proof}
Let $\{\gamma_t\}_{t \in S^1}$ be a generic homotopy
and let $C$ be the union of self-intersection points
of $\gamma_t$ for all $t$. Observe that, for a generic
homotopy $\gamma$, the set $C$ has Hausdorff dimension $1$.

Define $U \subset S^2 \setminus C$ to be the set of points such
that, for each $x\in U$, there are only finitely many times
$t \in \{t_1, ..., t_k \}$ when $\gamma_t$ intersects $x$. Moreover,
for each $t_i$ and $s_i$ with $\gamma_{t_i}(s_i) = x$, we require
that $\frac{ \partial \gamma}{\partial t} (t_i,s_i) \neq 0 $.
Since $\gamma$ is generic, $U$ is a set of full measure in $S^2$. 
Fix $x \in U$. 
For each moment of time when $\gamma_t$ passes through
$x$, define the local degree $d_i(x)$ to be the sign of the frame
$(\frac{ \partial \gamma}{\partial t} (t_i,s_i),
\frac{ \partial \gamma}{\partial s} (t_i,s))$.
The sum of all $d_x(i)$'s is then equal to the degree of
$\gamma_t$.

Consider a segment of the homotopy
$\gamma_t$ between $t_i$ and $t_{i+1}$. 
By a result of Whitney \cite{W} the turning number of
the composition of $\gamma_t$ with the stereographic projection $St_x(\gamma_t)$ does not change unless $\gamma_t$
undergoes a Type 1 move.  In this case, the turning number changes by $1$ (respectively $-1$) 
under a positive (respectively negative) Type 1 move.

When $t$ approaches $t_i$, the curve $St_x(\gamma_t)$
is contained in some large disc $D$ except for a simple arc $a$ which 
stretches to infinity. 
As $\gamma_t$ passes through
$x$ at time $t_i$, the arc $a$ is replaced by an arc $b$ as shown in Figure \ref*{fig:degree}.
We observe that the turning number
of $St_x(\gamma_t)$ decreases by $2$ if $d_x(i)$
is positive and increases by $2$ if $d_x(i)$
is negative. 

\begin{figure}
   \centering   
    \includegraphics[scale=0.25]{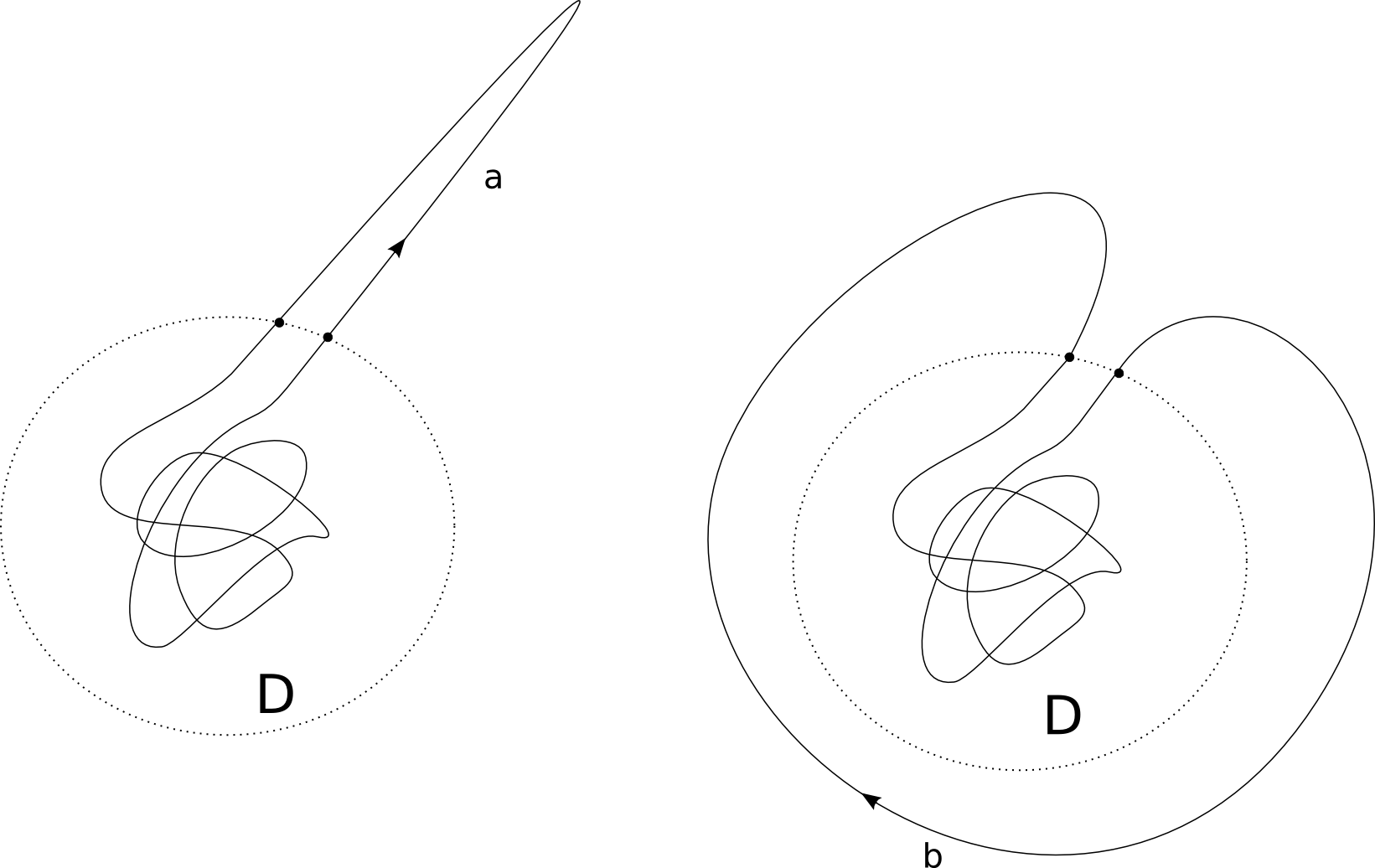}
    \caption{The degree of $St_x(\gamma_t)$ changes by $2$ when $\gamma_t$ passes through the point
    at infinity} \label{fig:degree}
\end{figure}

Since the homotopy is defined on $S^1$, it starts and ends on the 
same curve with the same turning number. 
It follows that the sum of signs of Type 1 moves equals
$2 \sum d_x(i)$.
\end{proof}

When $\gamma$ has no Type 1 moves this result simplified to the following statement.

\begin{corollary}
  \label{cor:regular_map}
  If $\gamma: S^1 \rightarrow \Lambda$ is a generic sweepout
  that contains no Type 1 Reidemeister moves, then $\gamma$
  is contractible.
\end{corollary}

\begin{remark}
Corollary \ref*{cor:regular_map} also follows from the fact that
the fundamental group of each connected component of the space of 
immersed curves is isomorphic to $\mathbb{Z}_2$ (see \cite{S}, \cite{I}, and \cite{T}).
Here's a sketch of the proof of this fact. Let 
$\gamma_0$ be an immersed curve and let $F_{\gamma_0}$
be the connected component of the space of immersed curves containing $F_{\gamma_0}$.
It follows from the parametric h-principle \cite{Gr}
that the space of immersed curves is weak homotopy equivalent to 
$\Lambda S M$, the space of free loops on the spherical tangent bundle $SM$ of $M$. It can be shown (see \cite{H}) that the group
$\pi_1 (F_{\gamma_0})$ is isomorphic to the centralizer 
of the class represented by $(\gamma_0 (s), \frac{d \gamma_0}{ds} (s))\in \Lambda S M$
in $\pi_1(S M, (\gamma_0 (0), \frac{d \gamma_0}{ds}(0))$. 
Since $S S^2 \cong SO_3 \cong \mathbb{R}P^3 $
we obtain that $\pi _1 (F_{\gamma_0})= \mathbb{Z}_2$.

The inclusion map $\iota: F_{\gamma_0} \rightarrow \Lambda M$
induces a homomorphism $\iota_{*}: \pi_1(F_{\gamma_0}) \rightarrow \pi_1(\Lambda M)$.
Since $\pi_1(F_{\gamma_0})= \mathbb{Z}_2$ and $\pi_1(\Lambda M) = \mathbb{Z}$,
$\iota_{*}$ must be trivial. 
\end{remark}

\subsection{Degree of a family containing constant curves}

From Corollary \ref*{cor:regular_map} we obtain 
the following characterization of the degree
of a family of curves that consists of immersed curves and constant curves.
This characterization will be important in the proof of our main
theorem.
Let $\gamma: S^1 \rightarrow \Lambda $ be a family of closed curves such that, for some closed interval
$I \subset S^1$, we have that $\gamma_t$ is a constant curve for all 
$t \in I$ and $\gamma_t$ is a generic homotopy
with no moves of Type 1 for $t \in S^1 \setminus I$.
We can show then that the family $\gamma$ has degree $0$
or $\pm 1$.  If we consider $\gamma$ on $S^1 \setminus I$, we get a map from an open interval of $S^1$ to $\Lambda$.
After a small perturbation, we may assume that this map is generic, and so if we choose points close to the endpoints of the
open interval, the corresponding curves are simple.  The degree of $\gamma$ then depends on the orientations of these curves.
This dependence is described as follows.  We assume that $\gamma$ has been perturbed slightly so the above property is true.

Fix an orientation on the sphere. 
As above, for a small $\delta>0$, let $t_1$ and $t_2$ be two points 
at the distance $\delta$ from the two endpoints of the open interval
$S^1 \setminus I$.  If $\delta$ is chosen to be sufficiently small, then for each $i=1,2$, the curve $\gamma_{t_i}$ 
is a simple closed curve  
bounding a small disc $D_i$. The orientation of the sphere
then induces an orientation of $D_i$, which in turn induces an orientation
on $\partial D_i$. If this orientation coincides with the
orientation of $\gamma_{t_i}$ we say that $\gamma_{t_i}$
is positively oriented and we say that it is negatively oriented 
otherwise. We say that $\gamma$ has the same 
orientation at the endpoints of $S^1 \setminus I$ if
for all sufficiently small $\delta > 0$
both $\gamma_{t_1}$ and $\gamma_{t_2}$ are oriented positively
or both are oriented negatively. Otherwise, we say that
$\gamma$ has different orientations at the endpoints of 
$S^1 \setminus I$.

\begin{corollary} \label{cor:deg_0}
Let $\gamma$ be as described above. If $\gamma$ has the same orientation
at the endpoints of $S^1 \setminus I$, then $\gamma$ has
degree $0$. If $\gamma$ has different orientations
at the endpoints of $S^1 \setminus I$, then $\gamma$ has
degree $\pm 1$.
\end{corollary}

\begin{proof}
Suppose first that $\gamma$ has the same orientation at the endpoints.
Define a new homotopy $\widetilde{\gamma}$ as follows.  For a sufficiently small $\delta > 0$, choose $t_1$ and $t_2$ as above.
On the interval $[t_1,t_2] \subset S^1 \setminus I$, we 
set $\widetilde{\gamma}$ to be equal to $\gamma$. On the complement
of this interval we define a homotopy from
$\gamma_{t_1}$ to $\gamma_{t_2}$ through short simple
closed curves as depicted in Figure \ref*{fig:degree_0}.
Observe that $\widetilde{\gamma}$ will have the same degree 
as $\gamma$, since two maps coincide
on $[t_1,t_2]$ and on the complement of $[t_1,t_2]$
the images of both maps have very small measure.
By Lemma \ref*{cor:regular_map} the degree of
$\widetilde{\gamma}$ is $0$.

\begin{figure}
   \centering   
    \includegraphics[scale=1.00]{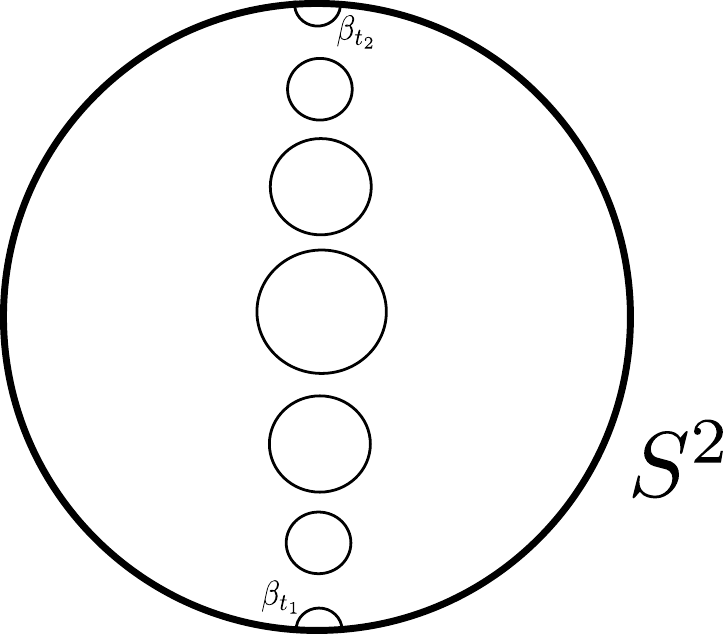}
    \caption{Connecting $\beta_{t_1}$ to $\beta_{t_2}$ by a sequence of simple closed curves} \label{fig:degree_0}
\end{figure}

Suppose $\gamma$ has different orientations at the endpoints.
As in the first case we will replace a portion
of $\gamma$ on the interval $I$ with a different
family of curves. Instead of a family of short curves
on Figure \ref*{fig:degree_0},
we will use a family that comes from a simple sweepout of the sphere.
Let $\beta$ be a simple sweepout of $S^2$ with the point
$\beta(-1)$ contained in the small disc $D_1$
and $\beta(1)$ contained in the small disc $D_2$.
Note that we do not require any control over the length or energy of 
curves in $\beta$ for this lemma, as the statement that we are proving 
is purely topological.

After a small perturbation we may assume that, for some
small $\epsilon >0$, the image of $\beta(-1+\epsilon)$
coincides with the image of $\gamma_{t_1}$ and
the image of $\beta(1-\epsilon)$
coincides with the image of $\gamma_{t_2}$.
Observe that $\beta$ has different orientation at the endpoints, and
so after a reparametrization we may assume that
$\gamma_{t_1}=\beta(-1+\epsilon)$ and
$\gamma_{t_1}=\beta(1-\epsilon)$.
We define a new family $\gamma'$ by setting
it to be equal to $\gamma$ on the interval
$[t_1,t_2]$, and to coincide with 
$\beta[-1+\epsilon,1-\epsilon]$ on the rest of $S^1$.
Since $\gamma'$ consists of immersed curves,
it must have degree $0$.
For any point $x\in S^2$ that does not lie in one 
of the small discs $D_1$ or $D_2$, there is exactly one preimage
of $x$ under $\beta$. Therefore, if $x$ is also a regular point
of $\gamma'$, then the total sum of signed preimages of $x$
under $\gamma$ restricted to $[t_1,t_2]$ must be equal to $\pm 1$.
We conclude that the degree of $\gamma$ is $\pm 1$.
\end{proof}

\section{Surgery on Sweepouts} \label{sec:surgery}
The purpose of this section is to perform surgery on sweepouts to prove the following proposition:

\begin{proposition}
  \label{prop:simple_sweepout}
  Given a generic sweepout $\gamma: S^1 \rightarrow \Lambda$ through curves of length (energy) less than $L$ (resp. $E$), we can find a sweepout $\widetilde{\gamma}$
  such that $\widetilde{\gamma}_0$ is a constant curve
  and $\widetilde{\gamma}$ consists of curves of length (energy) less than $L$ (resp. $E$).
\end{proposition}

To prove this theorem, we will require methods from the article \cite{CL} by the two authors.  This method begins with defining a certain graph.
Fix a map $\gamma:S^1 \rightarrow \Lambda$, and choose $t_1, \dots t_{n+1} \in S^1$ such that $t_i \neq t_j$ if $i \neq j$ and $1 \leq i,j \leq n$, exactly one Reidemeister move occurs between each
$t_i$ and $t_{i+1}$ for $i \in \{1, \dots, n \}$, and $t_1 = t_{n+1}$.  Note that at each time $t_i$, $\gamma_{t_i}$ is an immersed curve with transverse self-intersections, each of which consists
of exactly two arcs meeting at a point.

\subsection{Graph of self-intersections}

Form the graph $\Gamma$ as follows.  We will first describe how to add vertices, and then describe how to connect them with edges.

\vspace{2mm}

\noindent {\bf{Vertices}}
The vertices of $\Gamma$ will fall into $n+1$ sets $V_1, \dots, V_{n+1}$ and are defined as follows.  First of all, $V_{n+1}$ will simply be a copy of $V_1$.  Next, to construct $V_i$ for $i \in \{ 1, \dots, n \}$, we do the following.  If $\gamma_{t_i}$ is simple, then $V_i$ is empty.  If $\gamma_{t_i}$ is not simple, then every vertex in $V_i$ corresponds to a self-intersection of $\gamma_{t_i}$.

\vspace{2mm}

\noindent {\bf{Edges}}  To add edges to the graph, we do the following.  First, add an edge between each vertex in $V_1$ and each corresponding vertex
in $V_{n+1}$.  Next, for each $i \in \{1, \dots, n \}$, consider the Reidemeister move $R$ in $\gamma$ between $\gamma_{t_i}$ and $\gamma_{t_{i+1}}$.  We add edges on a case-by-case basis:
\begin{enumerate}
	\item	If $R$ is of Type 1, then we have two cases.  The first case is if $V_i$ or $V_{i+1}$ is empty; in this case we do not add any vertices to $\Gamma$.
		The second case is if either $V_i$ or $V_{i+1}$ has more than 1 vertex.  In this case, we do the following.  For every vertex in $V_i$
		that corresponds to a self-intersection $z$ of $\gamma_{t_i}$, either $z$ is deleted by the Reidemeister move $R$, in which case we don't add an edge to $z$, or we can follow $z$ forward to a self-intersection
		$z'$ of $\gamma_{t_{i+1}}$, in which case we join the vertex that corresponds to $z$ to the vertex that corresponds to $z'$ with an edge. 
	\item	If $R$ is of Type 2, then one of two things are true.  One possibility is that we can find two vertices in $V_i$ that correspond to two distinct self-intersection points $x$ and $y$ in $\gamma_{t_i}$
		that are deleted by $R$.  We then join these two vertices by an edge.  For every other vertex in $V_i$, that vertex corresponds to a self-intersection $z$ of $\gamma_{t_i}$.  We can follow this self-intersection
		forward in time to a self-intersection $z'$ of $\gamma_{t_{i+1}}$.  We join the vertex that corresponds to $z$ to the vertex that corresponds to $z'$ with an edge.
		
		The other possibility is that we can find two vertices in $V_{i+1}$ that correspond to self-intersection points $x$ and $y$ that
		were created by $R$.  In this case, we join these two vertices by an edge.   Additionally, for every other vertex in $V_{i+1}$, that vertex corresponds to a self-intersection point $z$ of $\gamma_{t_{i+1}}$.
		There is a self-intersection point $z'$ of $\gamma_{t_i}$ such that $z'$ can be followed forward to $z$.  We join the vertex that corresponds to $z'$ to the vertex that corresponds to $z$ with an edge.
	\item	If $R$ is of Type 3, then every vertex in $V_i$ corresponds to a self-intersection point $z$ of $\gamma_{t_i}$.  This self-intersection point can be followed forward to a self-intersection point $z'$ of
		$\gamma_{t_{i+1}}$.  We join the vertex that corresponds to $z$ to the vertex that corresponds to $z'$ with an edge.
\end{enumerate}

The following lemma characterizes the degree of each vertex of the graph $\Gamma$.

\begin{lemma}
  \label{lem:graph_degree}
  Each vertex in $\Gamma$ has degree $0$, $1$ or $2$.  Furthermore, using the above notation, consider any vertex $v \in V_i$ and the self-intersection $s$
  of $\gamma_{t_i}$ that corresponds to $v$.  Let $x, y \in \{ 0, 1 \}$ be defined as follows.  If $s$ is destroyed in a Type 1 deletion between $t_i$ and $t_{i+1}$, then $x = 1$, otherwise $x = 0$.
  If $s$ is created in a Type 1 creation between $t_{i-1}$ and $t_i$, then $y=1$, otherwise $y = 0$.  We then have that the degree of $v$ is $2 - x - y$.
  In the above statements, if $i = 1$, then $y = 1$, and if $i = n+1$, then $x = 1$.
\end{lemma}
\begin{proof}
  For each vertex $v \in V_i$, consider the Reidemeister move $R_1$ that occurs between $t_{i-1}$ and $t_i$, and let $R_2$ be the Reidemeister move that occurs between
  $t_i$ and $t_{i+1}$.  Again, if $i = 1$, then $R_1$ is not defined, and if $i = n + 1$, then $R_2$ is not defined.

  From the definition of the edges of $\Gamma$, we see that the edges added to $v$ (with corresponding self-intersection $s$ of $\gamma_{t_i}$) work exactly as follows.
  If $R_1$ is not defined or does not create $s$ in a Type 1 creation, then an edge is added to $\Gamma$ at $v$.  If $R_2$ is not defined or does not destroy $s$ in a Type 1 deletion,
  then a separate edge is added to $\Gamma$ at $v$.  This coincides with the degree computations in the statement of the lemma.
\end{proof}

We can look at the set of all vertices $\mathcal{V}$ such that a vertex $v \in \mathcal{V}$ corresponds to a vertex right after it was created by a Type 1 move,
or just before it is destroyed by a Type 1 move.  A corollary of this lemma is that $\mathcal{V}$ can be decomposed in a particular way.
\begin{corollary}
	\label{col:paths}
	The set $\mathcal{V}$ is equal to the union of a number of disjoint pairs of vertices such that, for each pair $(v, v')$, there is a path in the graph from $v$ to $v'$.  Furthermore,
	the path between a pair $(v,v')$ and the path between a different pair $(w,w')$ are completely disjoint (they do not share any edges).  Note that we may have that $v = v'$
	for a given pair $(v,v')$.
\end{corollary}

We will use such paths to generate our new homotopy.  
We first require a definition concerning how to cut a curve at a self-intersection.

\begin{definition}
	\label{defn:subcurve}
	Given a smooth curve $\alpha:S^1 \rightarrow S^2$ with isolated self-intersections, we say that a curve $\beta:S^1 \rightarrow S^2$ is a \emph{subcurve} of $\alpha$
	if there is some closed interval $[a,b] \subset S^1$ (possibly with $a = b$) such that $\beta$ is simply $\alpha$ restricted to $[a,b]$.  We have that $\alpha(a) = \alpha(b)$,
	so $\beta$ is all of $\alpha$, is a point, or $\alpha(a)$ is a self-intersection of $\alpha$.
\end{definition}


For each pair of self-intersections $(v,v')$ from Corollary \ref*{col:paths}, the self-intersection $s$ which corresponds to $v$ produces two subcurves, $C_{v, 1}$ and
$C_{v, 2}$.  To form $C_{v,1}$, begin at $s$ and follow the loop around according to its orientation
until we get back to $s$.  If we continue along the loop according to its orientation, we will encounter $s$ another time.  This forms the second subcurve $C_{v,2}$.
 Similarly, the self-intersection which corresponds to $v'$ produces two subcurves $C_{v', 1}$ and $C_{v', 2}$.
Since each edge in the graph corresponds to a continuous path between self-intersections, the path between $v$ and $v'$ produces a continuous path between the self-intersection
that corresponds to $v$ and the self-intersection that corresponds to $v'$.  This path then induces
a homotopy from $C_{v, 1}$ to $C_{v', i}$ for some $i \in \{1, 2 \}$, and it induces a homotopy from $C_{v, 2}$ to $C_{v', j}$, where $j \neq i$.
For consistency, assume that $C_{v, 1}$ and $C_{v', 1}$ correspond to the loops that were just created or are about to be destroyed by the appropriate Type 1 moves.  Let these
two homotopies be denoted by $h_{v, v', 1}$ and $h_{v, v', 2}$, respectively.

\begin{definition}

	\label{defn:good_pair}
	Given a pair $(v, v')$ as above, we say that it is \emph{good} if $h_{v, v', 1}$ ends at $C_{v', 2}$ ($h_{v, v', 1}$ starts at $C_{v, 1}$ by definition).
\end{definition}

If a good pair $(v, v')$ exists, we can 
contract any curve in the sweepout to a point through curves of controlled length and energy.

\begin{lemma} \label{lem:contraction}
If the graph $\Gamma$ contains a good pair, then for any 
curve $\gamma_t$ in the sweepout, there is a contraction 
of $\gamma_t$ to a point through curves of length (energy)
less than $L$ (resp. $E$).
\end{lemma}

\begin{proof}
Observe that we can use the Type 1 destruction or creation 
of a small loop which corresponds to the terminal vertex $v'$ to homotope $C_{v',2}$ to 
a curve in the sweepout $\gamma$. Denote this curve by $\gamma_{t^*}$.
Similarly, we can use the Type 1 move corresponding
to the initial vertex $v$ to homotope the curve $C_{v,1}$ to a point.
The fact that $(v,v')$ is a good pair means that
$C_{v,1}$ and $C_{v',2}$ are homotopic via the homotopy $h_{v, v', 1}$
through subcurves of curves in $\gamma$. 
Hence, $\gamma_{t^*}$ (and every other curve in $\gamma$) is
homotopic to a point through subcurves of curves in $\gamma$.
\end{proof}

We now have that, if there is a good pair $(v,v')$, then our proof of Proposition \ref*{prop:simple_sweepout} is true.  
To see this let the homotopy $\alpha_t$ be a contraction
of a curve in $\gamma$. Define $\widetilde{\gamma}= (-\alpha) * \gamma * \alpha $,
where $- \alpha$ signifies $\alpha$ in the reverse direction.
This family of curves has the same degree as $\gamma$ and starts at a constant curve.
By Lemma \ref*{lem:contraction} we can choose $\alpha$ so 
that the lengths of curves are less than $L$ and energy less
than $E$.
The next subsection proves the existence of such a good pair for
any sweepout $\gamma$.

If $(v, v')$ is not a good pair, then the ``cutting" procedure as above produces two homotopies, at least one of which is nontrivial.
One is tempted to apply both of these cutting procedures
repeatedly in the hope of constructing
a homotopy with the desired properties.
However, even for a simple case of homotopies of curves
with at most $2$ self-intersections it may happen
that the maximal number of self-intersections of curves in the new homotopy
does not decrease, no matter how many times we apply the cutting
procedure. In fact, there are situations in which a portion of the homotopy
is replicated every time we apply the cutting procedure.
This is illustrated in Figure \ref*{fig:replicating}.

\begin{figure}
   \centering   
    \includegraphics[scale=0.35]{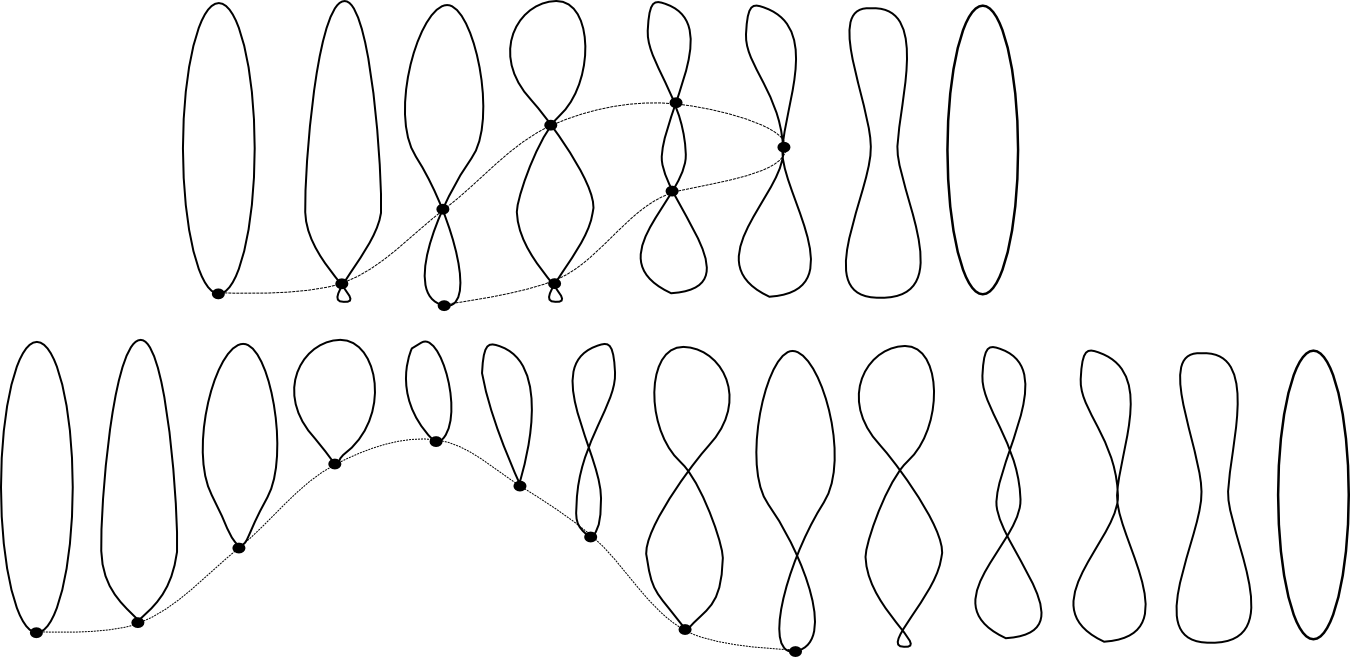}
    \caption{We cannot remove self-intersections by repeatedly applying the cutting procedure} \label{fig:replicating}
\end{figure}

The top picture describes the original homotopy.
After we apply the cutting procedure we obtain two new homotopies.
One of these homotopies will start at a point and end at a point,
but it may happen that it has degree $0$.
Then we have to consider the other homotopy,
which is shown on the bottom of Figure \ref*{fig:replicating}.
This homotopy has two Type 1 creations exactly like
the original homotopy, so we cannot get rid of these
intersections by applying the cutting procedure again.

\subsection{Existence of a good pair $(v, v')$}

Assigning to each Type 1 move a sign as we did when defining the degree formula 
in Proposition \ref*{degree formula} (using some appropriate $x \in S^2$), we see that since the degree of $\gamma$ is odd, 
the sum of all of the signs of all of the Type 1 moves is $2~mod~4$.  
Hence, we can find a pair
$(v, v')$ such that the sign of the Type 1 move associated to $v$ is positive, as is the sign
of the Type 1 move associated with $v'$. Note that if $v = v'$, then there is some ambiguity as to which Type 1 moves we are referring to.
In this case, there is a self-intersection that is created by a positive Type 1 move,
and then which is immediately destroyed by a positive Type 1 move.  In this case, clearly $(v,v')$ is a good pair, and these are the two Type 1 moves
to which we are referring.  For the remainder of this section, we may thus assume that $v \neq v'$.  We then have the following.

\begin{lemma}
If $v$ and $v'$ correspond to Type 1 moves
of the same sign then the pair $(v, v')$ is a good pair.
\end{lemma}

\begin{proof}
Note that a positive Type 1 move can be either
a creation of a small positively oriented loop
or a destruction of a small negatively oriented loop.
Each point on the graph $\Gamma$ corresponds to
a transverse self-intersection point of $\gamma_t$ for some $t$.
By cutting $\gamma_t$ at the self-intersection point
we obtain two connected curves $\gamma_t^A$ and $\gamma_t^B$.
To each point on the graph and a choice of $\gamma_t^A$ or $\gamma_t^B$,
we will associate a binary invariant $L$ which we will call the local orientation.
This invariant is based only on local data in the neighborhood of
the self-intersection.
We will show that the invariant changes sign every time 
the path in the graph between $v$ and $v'$ changes direction.

Let $p$ be a self-intersection point of $\gamma_t$.
Without any loss of generality we may assume that
two arcs intersect at $p$ perpendicularly.
We cut $\gamma_t$ at $p$ and smooth out the intersection in such a way that we obtain
two connected curves that inherit their orientations from
$\gamma_t$. Let $Q$ be a small disc in the neighborhood of $p$.
After smoothing, the curve separates $D$ into three connected components
(see Figure \ref*{invariant}). Let $A$ and $B$ denote two components
that do not share a boundary and let 
$\gamma_t^A$ and $\gamma_t^B$ denote curves adjacent to 
$A$ and $B$ correspondingly. Let $v_1$ be a tangent vector of
$\gamma_t^A$ at $p$ (solid line on Figure \ref*{invariant}) and
let $v_2$ be a vector
pointing from $p$ to a point on $\partial A \cap \partial D$
(dashed line on Figure \ref*{invariant}).
We define $L(\gamma_t^A) = 1$ if the ordered pair $(v_1,v_2)$
is positively oriented and $L(\gamma_t^A) = -1$ otherwise.  We
define $L(\gamma_t^B)$ in the same manner.
Observe that $L(\gamma_t^A) = - L(\gamma_t^B)$.

\begin{figure}
   \centering	
	\includegraphics[scale=0.9]{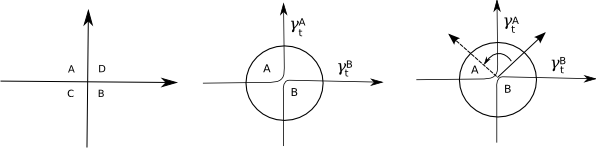}
	\caption{Definition of $L(\gamma_t^A)$}  \label{invariant}
\end{figure}


For each point $q$ in the path $P$ between $v$ and $v'$, let
$L_q$ denote $L(\gamma_t^A)$, where $\gamma_t^A$ corresponds
to the appropriate curve in $h_{v, v', 1}$.  

We observe that $L_q$ 
does not change when $\gamma$ undergoes a Type 3 move
or a move that does not involve the self-intersection that we follow to form the path $P$. 
Consider a segment of the path $P$ corresponding
to a part of the homotopy where $\gamma_t$ goes through
a Type 2 move involving the self-intersection that we follow to form $P$.
On the graph $\Gamma$, this looks like a change of direction
of the path $P$ (see Figure \ref*{path}).
Let $q_1$ be a point on the path just before the change of direction 
and let $q_2$ be a point just after it.
We claim that $L_{q_1} = - L_{q_2}$.
This follows by considering Figure \ref*{type2}.

\begin{figure}
   \centering	
	\includegraphics[scale=0.25]{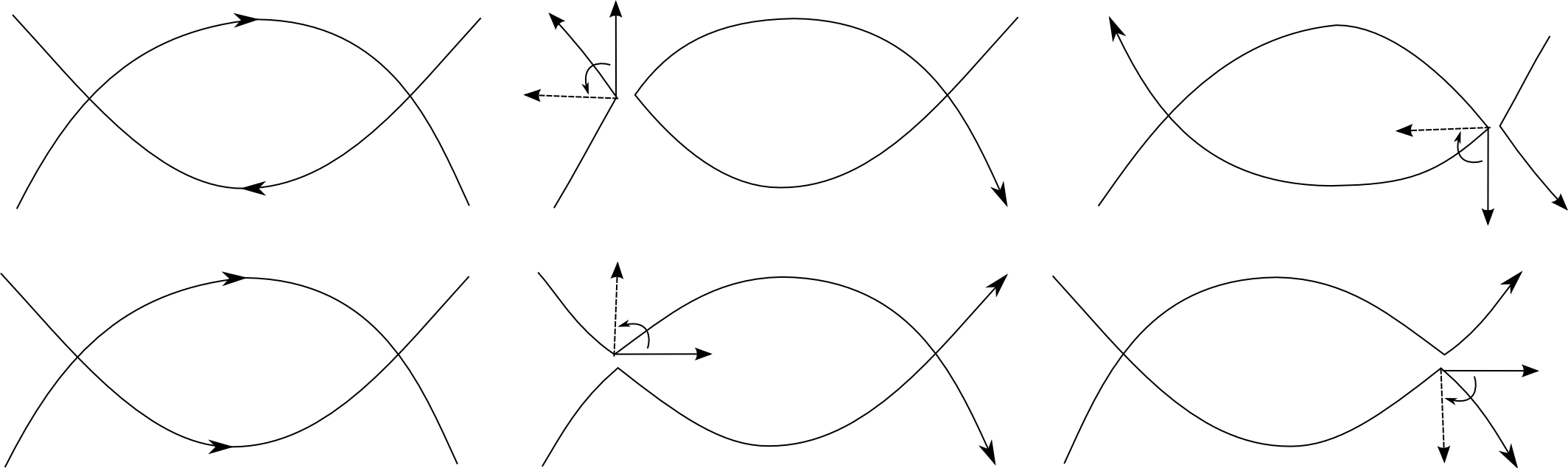}
	\caption{$L_q$ changes sign under Type 2 moves}  \label{type2}
\end{figure}

With this result we can now prove the lemma. Suppose $(v,v')$ is not a good pair.

Consider two cases. Suppose first that the path $P$
changes direction an odd number of times as in
Figure \ref*{path}(a).  It follows that the Type 1 move corresponding to $v$ and
the Type 1 move corresponding to $v'$ are either both creations, or are both destructions.
In either case, the orientation of the loop that is being created or destroyed 
at $v$ is different than the orientation of the loop that is being created or destroyed at $v'$
since $L$ changes sign from the beginning of $P$ to the end of $P$ and $(v,v')$ is not a good pair.
Thus, the sign of the Type 1 move associated with $v$ is opposite to the sign of the Type 1 move
associated with $v'$, which is a contradiction.

\begin{figure}
   \centering	
	\includegraphics[scale=1]{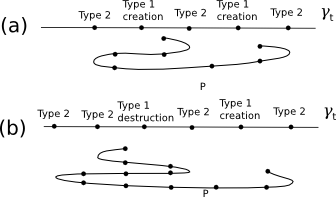}
	\caption{Path $P$ may change direction an even or odd number of times} \label{path}
\end{figure}

Suppose $P$
changes direction an even number of times as in
Figure \ref*{path}(b). Then either the Type 1 move associated with $v$ is a creation and 
the Type 1 move associated with $v'$ is a destruction, or the move associated with $v$ is a destruction
and the move associated with $v'$ is a creation.  In either case, the orientation of the small loop
that is created or destroyed at $v$ is the same as the orientation of the small loop that is being created or destroyed
at $v'$, since the $L$ invariant has the same value at the beginning
and at the end of $P$, and since $(v,v')$ is not a good pair.  This implies that the sign of the Type 1 move at $v$
is opposite to that of the Type 1 move at $v'$, which is a contradiction.  

\end{proof}

As described above, this implies that Proposition \ref*{prop:simple_sweepout} is true.

\section{Proof of Theorem \ref*{thm:simple_sweepout}}
Given a sweepout $\gamma$, we apply Proposition \ref*{prop:simple_sweepout} to it
to obtain a sweepout $\widetilde{\gamma}$ that starts and ends on a constant 
curve. For simplicity we will think of $\widetilde{\gamma}$
as a family of curves defined on $[0,1]$ with
$\widetilde{\gamma}_0=\widetilde{\gamma}_1$ being the constant curve.
It follows from the construction that for all $t$ sufficiently close
to the endpoints $\widetilde{\gamma}_t$ is a simple closed curve, and
$\widetilde{\gamma}_t$ is not constant for $t \in (0,1)$.

Our goal now is to modify $\widetilde{\gamma}$ 
so that it consists of curves without self-intersections, satisfies the
desired length and energy bounds and has an odd degree (in fact, we will show that 
the new family has degree $\pm 1$).
To do this we use the results of \cite{CL}.
Let $t_1=\delta$ and $t_2=1-\delta$ be two points near the endpoints of $[0,1]$ so that  $\widetilde{\gamma}_{t_1}$ and
$\widetilde{\gamma}_{t_2}$
are simple closed curves each bounding a small closed disc,
denoted by $D_1$ and $D_2$
respectively. We perturb 
the homotopy slightly so that discs $D_1$ and $D_2$ are disjoint.
If $\delta$ is sufficiently small we can always do this so that
the length (energy) of curves never exceed $L$ (resp. $E$).

We can now apply Theorem $1.1'$ from \cite{CL} to the homotopy $\widetilde{\gamma}[t_1, t_2]$.  This produces an isotopy
$\alpha$ which starts at $\widetilde{\gamma}_{t_1}$ and ends at $\pm \widetilde{\gamma}_{t_2}$,
where $-\widetilde{\gamma}_{t_2}$ denotes $\widetilde{\gamma}_{t_2}$ with the opposite orientation.  
Curves $\alpha_0$ and $\alpha_1$ are simple curves of some very small energy.
We can contract each of them to a point in the corresponding small disc $D_i$
through short simple closed curves. Hence, we can turn $\alpha$ into a map
$\beta: S^2 \rightarrow S^2$.



\begin{lemma} \label{lem:deg_1}
Map $\beta$ has degree $\pm 1$.
\end{lemma}

\begin{proof}

First, we show that the degree of $\beta$ must be odd.

Let $x \in S^2$ be a regular point of $\widetilde{\gamma}$, which does not intersect 
any of the curves $\widetilde{\gamma}_{t}$ for $t$ 
in one of the small intervals $[0,t_1]$ and $[t_2,1]$,
and which also does not lie on any of the self-intersection
points of curves in $\widetilde{\gamma}$. By selecting $t_1$ and $t_2$
sufficiently close to $0$ and $1$, we can ensure that the set of such points has
nearly full measure in $S^2$.
In addition we require that $x$ does not lie on any curve of $\widetilde{\gamma}[t_1, t_2]$ at which a Reidemeister move
occurs, that is, any curve
which has a singularity, a tangential self-intersection or an intersection which involves three arcs.
There will be only finitely many such curves in the homotopy (see Subsection
\ref*{sec:generic}).



For each $t \in [t_1, t_2]$, the curve $\alpha_t$
is a redrawing of some curve $\widetilde{\gamma}_{t'}$, where $t' \in [t_1,t_2]$, 
but may differ from $t$.
We can modify the original homotopy $\widetilde{\gamma}_{t}$ so that
$\alpha_t$ is a redrawing of $\widetilde{\gamma}_{t}$ for each $t$.
This is done simply by making the homotopy move back and forth
through the same curves multiple times for certain subintervals of the parameter space.
This does not change the sum of the signs of the signed preimages of $x$ with respect to $\widetilde{\gamma}$, and does not
affect the initial and final curves. 
Since the total number of preimages of $x$ under $\widetilde{\gamma}$ is
odd (as the sum of the signed preimages is $1$), the same must be true about $\alpha$, and so the sum of all of the signed preimages of $\alpha$ must be odd as well.
Since $\beta$ consists of constant curves and simple closed curves,
by Corollary \ref*{cor:deg_0}, $\beta$ has degree $\pm 1$.

\end{proof}

We now have that $\alpha: [0,1] \rightarrow S^2$ consists
of simple curves (except for $\alpha_0$ and $\alpha_1$ which are 
constant), but different curves in $\alpha$ may intersect.

\section{Applications to converting homotopies to isotopies in an effective way}

In this section we use the methods of this paper to prove
Theorem \ref*{thm:isotopy}. This result was conjectured to be true
by the authors in (\cite{CL}, remark in the end of Section 3).

We wish to show that if two simple curves $\gamma_0$ and $\gamma_1$ on a Riemannian
manifold $(M,g)$ are homotopic through curves of length less than $L$, then

\begin{enumerate}
	\item	$\gamma_0$ and $\gamma_1$ are homotopic through simple curves of length less than $L$, or
	\item	$\gamma_0$ and $\gamma_1$ are each contractible through simple closed curves of length less than $L$.  Here, by simple curves, we
		mean that all curves except for the final constant curve are simple.
\end{enumerate}

We break the proof into two cases. If $\gamma_0$ is non-contractible, then since $\gamma_1$ is homotopic to $\gamma_0$, it is also non-contractible.
In \cite{CL} it was shown that $\gamma_0$ and $\gamma_1$ are isotopic
through curves of length less than $L$.

The case that we deal with now is if $\gamma_0$ (and $\gamma_1$) are contractible.  We can assume that the images of $\gamma_0$ and $\gamma_1$ do not coincide by perturbing
our homotopy slightly.  If we can show that the result holds for the perturbed curves, then we can perturb the curves back to the original curves, proving the theorem.

Let $M$ be any surface and consider the universal cover $\widetilde{M}$ of $M$.  The universal cover is either $S^2$ or a contractible subset of $\mathbb{R}^2$.
Since the curves $\gamma_0$ and $\gamma_1$ are contractible we can lift the homotopy between them
to a homotopy in the universal cover.
Hence, without any loss of generality, we may assume that
$M$ is either a plane or a sphere.

Suppose $M$ is a contractible subset of $\mathbb{R}^2$.  In this case, we can define the turning number $\mathcal{T}_0$ of $\gamma_0$ and the turning number $\mathcal{T}_1$ of
$\gamma_1$ as before.  Since both $\gamma_0$ and $\gamma_1$ are simple, $\mathcal{T}_0 = \pm 1$, and $\mathcal{T}_1 = \pm 1$.  We can apply the procedure described in \cite{CL}
to produce a homotopy $H: [0,1] \times S^1 \rightarrow \mathbb{R}^2$ through simple curves of length less than $L$
from $\gamma_0$ to $\pm \gamma_1$.  Since this map goes through simple closed curves,
the turning number of $\gamma_0$ and the turning number of $H(1 \times S^1)$ agree.  Hence, if $\mathcal{T}_0 = \mathcal{T}_1$, we are done.

If $\mathcal{T}_0 \neq \mathcal{T}_1$, then 
by an argument analogous to the one used to prove Proposition \ref*{degree formula}, the total sum of signs of Type 1 moves in
$\gamma$ is $\pm 2$. Hence, as in Section 3, we can construct a graph $\Gamma$ and find a good pair of vertices $(v,v')$ connected by a path $P$ in the graph.
We can cut homotopy $\gamma$ along $P$ to obtain a contraction
of $\gamma_0$ through curves of length less than $L$.  Using Theorem $1.1'$ from \cite{CL}, we can turn this into a contraction through simple closed curves
of length less than $L$.  We can do the same with $\gamma_1$.  This completes the proof of our theorem in the case when $M= \mathbb{R}^2$.

Now, suppose that $M = S^2$.  We begin with a lemma:

\begin{lemma}
	\label{lem:degree}
	If we consider the map $f: [0,1] \times S^2$ defined by 
	$\gamma(t,s)$, then there is a regular point $p \in S^2$ such that $f^{-1}(p)$ contains an even number of points.
\end{lemma}

\begin{proof}
 To see this we can find a smooth map
$g: [-1, 2] \times S^1$ such that

\begin{enumerate}
	\item	$f(t,s) = g(t,s)$ for $t \in [0,1]$.
	\item	$g(t,s)$ on $[-1,0]$ goes from a constant curve to $\gamma_0$.
	\item	$g(t,s)$ on $[1, 2]$ goes from $\gamma_1$ to a constant curve.
	\item	$S^2 \setminus (g([-1,0] \times S^1) \cup g([1,2] \times S^1))$ contains an open set $U \subset S^2$. 	
	\item	There exists an open set $V \subset S^2$  such that $V \subset  (g([-1,0] \times S^1) \cup g([1,2] \times S^1))$ and for every $x \in V$, 
		there is exactly one pair $(t,s) \in [-1,0] \times S^1 \cup [1,2] \times S^1$ that maps to $x$.
\end{enumerate}

$g$ is constructed as follows.  To construct $g$ on $[-1,0]$, we choose one of the discs bounded by $\gamma_0$, and contract $\gamma_0$ to a point in the disc in a monotone way.  We repeat the same procedure with $\gamma_1$.  Since $\gamma_0$ and $\gamma_1$ do not have the same image, all of the above
properties are satisfied.

Since $g$ maps $2 \times S^1$ to a point and $-1 \times S^1$ to a point, we can define a corresponding map from $S^2$ to $S^2$. 
If this degree is even, then for every regular value $p$ in $U$,
the total number of points in $f^{-1}(p)$ is even, as the sum of all of the signed preimages is equal to the degree of $g$, which is even.
If this degree is odd, then for every regular value in $p \in V$, the 
total number of preimages of $p$ ($f^{-1}(p)$) is even, since it is equal to the
degree of $g \pm 1$.  This completes the proof of our claim.
\end{proof}

We can now prove our theorem.  From the above lemma, choose a regular point $p \in S^2$ whose preimage has an even number of points.  Consider
the stereographic projection $St_p$ from $S^2 \setminus \{ p \}$ to $\mathbb{R}^2$.  We can additionally assume that $p$ is not in the image of 
any curve in $\gamma$ at which a Reidemeister move occurs (that is, any curve in $\gamma$ that has a singularity, non-traverse intersection
or an intersection which involved three arcs) and that $p$ is not in the image of
$\gamma_0$ or $\gamma_1$.  
Let $\widetilde{\gamma}$ denote the isotopy from Theorem $1.1'$ of \cite{CL}
that starts on $\gamma_0$ and ends on $\pm \gamma_1$.
As remarked in the proof of Lemma \ref*{lem:deg_1} we can assume that  
$\widetilde{\gamma}$ is a redrawing of the curve $\gamma_t$ for each
$t$. Let $\mathcal{S}_t$ denote the difference between the turning
number of $St_p(\gamma_t)$ and the turning number of 
$St_p(\widetilde{\gamma_t})$.  Note that $\mathcal{S}_t$ is undefined
for finitely many times $t$ when the curve $\gamma_t$
(and, as a result, $\widetilde{\gamma_t}$) intersects $p$.
Since both homotopies start on the same curve we have $\mathcal{S}_0=0$.
When $\gamma$ goes through a positive Type 1 move
its turning number increases by $1$, while the turning number
of the corresponding curve in $\widetilde{\gamma}$ remains the same.
Hence $\mathcal{S}_t$ increases by one. Similarly, it decreases by
one whenever $\gamma$ goes through a negative Type 1 move.
If $\gamma$ passes through the point $p$
then, as in the proof of Proposition \ref*{degree formula}
(see Figure \ref*{fig:degree}), we have that the turning number
of $\gamma$ changes by $\pm 2$, as does the turning number
of $\widetilde{\gamma}$.
Hence, every time the curve goes through the point $p$, we have that 
$\mathcal{S}_t$ changes by $0$ or $\pm 4$.

In the end we have two possibilities. 
If $\mathcal{S}_1 = 0$, 
then $\widetilde{\gamma}$ defines the desired homotopy and we are done.
If $\mathcal{S}_1 = \pm 2$, 
then it follows that the total sum of signed Type 1 moves in homotopy $\gamma$ 
is $\pm 2$.
Hence, there is a path between a good pair of vertices $(v, v')$ in the appropriate graph $\Gamma$, and so
using the methods from Section 4, we can contract $\gamma_0$ to a point through simple  curves of length less than $L$.  
Similarly, we can do this for $\gamma_1$.  This completes the proof.

\end{document}